\documentclass[12pt,thmsa]{article}
\usepackage{amssymb}
\usepackage{amsfonts}
\usepackage{cite}
\usepackage{amsmath}
\usepackage{color}

\numberwithin{equation}{section}
\newtheorem{theorem}{Theorem}[section]

\newtheorem{corollary}[theorem]{Corollary}

\newtheorem{definition}[theorem]{Definition}

\newtheorem{lemma}[theorem]{Lemma}

\newtheorem{proposition}[theorem]{Proposition}
\newtheorem{remark}[theorem]{Remark}

\newenvironment{proof}[1][Proof]{\noindent\textbf{#1.} }{\hfill $\square$}

\begin{document}

\title{Positive solutions for nonlinear Schr\"{o}dinger--Poisson Systems with general nonlinearity}
\date{}
\author{Ching-yu Chen\thanks{%
E-mail address: aprilchen@nuk.edu.tw (C.-Y. Chen)} and Tsung-fang Wu\thanks{%
E-mail address: tfwu@nuk.edu.tw (T.-F. Wu)} \\
Department of Applied Mathematics \\
National University of Kaohsiung, Kaohsiung 811, Taiwan}
\maketitle

\begin{abstract}
In this paper, we study a class of Schr\"{o}dinger-Poisson (SP) systems with general nonlinearity where the nonlinearity does not require Ambrosetti-Rabinowitz and Nehari monotonic conditions. We establish new estimates and explore the associated energy functional which is coercive and bounded below on Sobolev space. Together with Ekeland variational principle, we prove the existence of ground state solutions. Furthermore, when the `charge' function is greater than a fixed positive number, the (SP) system possesses only zero solutions. In particular, when `charge' function is radially symmetric, we establish the existence of three positive solutions and the symmetry breaking of ground state solutions.
\end{abstract}

\newline


\section{Introduction}

In this paper, we consider the following Schr\"{o}dinger--Poisson (SP)
system with general nonlinearity:
\begin{equation}
\left\{
\begin{array}{ll}
-\Delta u+u+\rho \left( x\right) \phi u=f\left( u\right) & \text{ in }%
\mathbb{R}^{3}, \\
-\Delta \phi =\rho \left( x\right) u^{2} & \ \text{in }\mathbb{R}^{3},%
\end{array}%
\right.  \tag*{$\left( SP_{\rho }\right) $}
\end{equation}%
where $\rho \in C\left( \mathbb{R}^{3},\mathbb{R}\right) $ and $f\in C\left(
\mathbb{R},\mathbb{R}\right) $ satisfying the assumptions that

\begin{itemize}
\item[$\left( D1\right) $] $\rho (x)$ is positive on $\mathbb{R}^{3}$ and $%
\lim_{\left\vert x\right\vert \rightarrow \infty }\rho \left( x\right) =\rho
_{\infty }>0\ $uniformly on$\ \mathbb{R}^{3};$

\item[$\left( F1\right) $] $f(s)$ is a continuous function on $\mathbb{R}$
such that $f(s)\equiv 0$ for all $s<0$ and $\lim_{s\rightarrow 0^{+}}f\left(
s\right) =0;$

\item[$\left( F2\right) $] there exists $2\leq q<3$ and $a_{q}>1$ if $%
q=2;a_{q}>0$ if $2<q<3$ such that%
\begin{equation*}
\lim_{s\rightarrow \infty }\frac{f\left( s\right) }{s^{q-1}}=a_{q}.
\end{equation*}
\end{itemize}

System $\left( SP_{\rho }\right) $ can be used to describe the interaction
of a charged particle with the electrostatic field in quantum mechanics. The
unknown $u$ is the wave function associated with the particle and $\phi $
the electric potential, whereas $\rho :\mathbb{R}^{3}\rightarrow \mathbb{R}%
^{+}$ is a measurable function representing a `charge' corrector to the
density $u^{2}$ with the nonlinear function $f\left( u\right) $ representing
the interaction effect among many particles. For more detailed description
on the physical aspects of the problem, we refer the readers to \cite{BF,R1}.

It is well known that such an SP system can be transformed into a nonlinear
Schr\"{o}dinger equation with a non-local term when $\rho \in L^{\infty
}\left( \mathbb{R}^{3}\right) \cup L^{2}\left( \mathbb{R}^{3}\right) $\cite%
{A,R1}. Briefly, by the Lax--Milgram theorem, for all $u\in H^{1}(\mathbb{R}%
^{3}) $, there exists a unique $\phi=\phi _{\rho,u}\in D^{1,2}(\mathbb{R}%
^{3})$ satisfying $-\Delta \phi =\rho (x)u^{2}$ with
\begin{equation}
\phi _{\rho ,u}(x)=\frac{1}{4\pi }\int_{\mathbb{R}^{3}}\frac{\rho (y)u^{2}(y)%
}{|x-y|}dy.  \label{1-2}
\end{equation}%
Subsequently, system $(SP_\rho)\,$ is rewritten as
\begin{equation}
\begin{array}{ll}
-\Delta u+u+\rho \left( x\right) \phi _{\rho ,u}u=f\left( u\right) & \text{
in }\mathbb{R}^{3},%
\end{array}
\tag*{$\left( E_{\rho }\right) $}
\end{equation}%
which can now be studied in a variational setting with the solutions given
by the critical points of the corresponding energy function $J_{\rho }:H^{1}(%
\mathbb{R}^{3})\rightarrow \mathbb{R}$ defined as%
\begin{equation*}
J_{\rho }\left( u\right) =\frac{1}{2}\int_{\mathbb{R}^{3}}\left( |\nabla
u|^{2}+u^{2}\right) dx+\frac{1}{4}\int_{\mathbb{R}^{3}}\rho \left( x\right)
\phi _{\rho ,u}u^{2}dx-\int_{\mathbb{R}^{3}}F\left( u\right) dx,
\end{equation*}%
where $F\left( u\right)=\int_{0}^{u}f\left( s\right)ds. $ Thus $(u,\phi )\in
H^{1}(\mathbb{R}^{3})\times D^{1,2}(\mathbb{R}^{3})$ is a solution of system
$\left( SP_{\rho }\right) $ if and only if $u$ is a critical point of the
functional $J_{\rho }$ and $\phi =\phi _{\rho,u}.$ Moreover, $(u,\phi )$ is
a ground state solution of system $(SP_{\rho })$, if $u$ is a ground state
solution of equation $(E_{\rho })$ with the corresponding functional given
by $J_\rho$.

A vast amount of work has been dedicated to the study of SP systems with the
local nonlinearity $f(u)$ mostly given in a pure power representation, e.g. $%
f(u)=|u|^{p-2} u$, and others in a more generalized formulation \cite%
{AR,AA,AdP,AL,CM1,CM2,DTWZ,IV,SW,SWF,SWF1,SWF2,ZLZ,ZZ,ZH1,ZH2,WZ}, in an
autonomous setting with $\rho (x)$ being constant or otherwise.
In these settings, the lack of compactness is a common issue and presents a
major difficulty for the application of standard variational techniques.
Ruiz \cite{R1} overcame this issue by restricting the functional in a
radially symmetric space when studying a class of autonomous SP systems with
$\rho (x)=\sqrt{\lambda}>0$ and $f(u)=|u|^{p-2} u$. The existence of
positive radial solutions was proven by minimizing the associated energy
functional on a certain manifold that is defined using a combination of
Nehari and Pohozaev equalities for the case when $3<p<6$, whereas for $2<p<3$%
, Strauss inequality was used to show the boundedness of Palais-Smale (PS)
sequences, yielding a positive radial solution for small value of $\lambda$;
however, for $\lambda\ge 1/4$, $u=0$ is the unique solution.
Azzollini-Pomponio \cite{AP} and Zhao-Zhao \cite{ZZ} improved on these
results by showing the existence of ground state solutions (possibly
non-radial) when $\lambda =1$ and $3<p<6$ using the Nehari-Pohozaev manifold
developed by Ruiz\cite{R1}.

For the non-autonomous system involving a non-constant charge density $%
\rho(x)$, most studies showed the existence of solutions for $4<p<6$ with $%
f(u)=|p|^{p-2}u$. For examples, Cerami-Varia \cite{CV}, were able to prove
the existence of ground state and bound state solutions for $4<p<6$ without
imposing symmetry assumptions by establishing a compactness lemma and using
the Nehari manifold method with $\rho(x) $ satisfying some prescribed
conditions. However, for the case when $2<p\leq 4,$ the Palais-Smale (PS)
condition on $H^{1}(\mathbb{R}^{3})$ remained unresolved so that the
associated energy functional is not bounded below on both Nehari manifold $%
\left( 2<p\leq 4\right) $ and Nehari-Pohozaev manifold $\left( 2<p<3\right) $
for $\left\Vert \rho \right\Vert _{\infty }$ sufficiently small. The
commonly adopted variational methods are rendered insufficient in these
situations. In \cite{SWF1}, an alternative approach was proposed, using a
novel constraint technique, the authors demonstrated the existence of
positive solutions including ground state solutions for the case when $%
2<p\leq 4$ and $\left\Vert \rho\right\Vert _{\infty }$ sufficiently small,
thus filling the gap in \cite{CV}; in particular, the existence result of
ground state solutions was obtained for $3.1813\approx \frac{1+\sqrt{73}}{3}%
<p\leq 4$. This result is improved further by Mercuri and Tylerin \cite{MT}
in their recent paper where the existence of ground state solutions was
demonstrated for $3<p<4$ with different hypotheses (coercive and
non-coercive) on the behavior of $\rho$ at infinity.

SP systems involving a more general nonlinearity have been studied, for
examples, within the context that satisfies Ambrosetti--Rabinowitz (A-R)
condition \cite{ZZ,ZH1} thus ensuring the mountain pass geometry of the
associated energy functional and guarantee the boundedness of Palais--Smale
sequences. Existence results with superlinear $f(u)$ without the AR
condition are also obtained, typically with $f(s)<C_1s+C_2s^5$ \cite{ZH2} or
$f(s)/|s|^3$ being an increasing function for $|s|>0$ \cite{AL,SW,ZH2}. In
these problems, alternative techniques were adopted, for examples, using the
argument from \cite{BJ} or otherwise in order to demonstrate the boundedness
of PS sequences. Systems with a Berestycki \& Lion type nonlinearity were
also investigated \cite{AA,AdP}, using a concentration and compactness
argument in \cite{AA}, a non-radial solutions was proven to exist through
the minimization of a modified functional. General linearity that is
asymptotically linear at infinity has also been studied in \cite{WZ} where
bounded PS sequences were obtained without requiring the AR condition and a
positive solution was found when the charge function $\rho(x)$ is a small
constant.
As we have mentioned earlier, related works concerning the various forms of
SP systems are numerous and we cannot hope to comment on them all here (nor
can our comments be in sufficient details); we have therefore limited our
brief remarks on those that are most relevant to our study here. The
interested readers should find sufficient materials from afore mentioned
references and the references therein for further details.

In the present paper, we focus our attention on the existence and symmetry
of ground state solutions for system $(SP_{\rho })$ subject to the
conditions $(D1), (F1)$ and $(F2),$ extending our previous study in \cite{W}
where non-radial ground state solutions were obtained using Nehari manifold
for the SP system with $\rho(x)\equiv\sqrt{\lambda}>0 $ and nonlinearity $%
f(u)=a(x)|u|^{p-2}u\,$ for $2<p<3$. The challenge presents us here is that
within the context of these weaker conditions, the variational approaches
such as Nehari and Nehari--Pohozaev manifolds are no longer applicable nor
can we assume the mountain pass geometry of the energy functional to
construct bounded PS sequences without the AR condition being satisfied. We
thus make progress below using an entirely different approach by means of a
global minimizer, the results are summarized as follows. For the autonomous
cases with $\rho(x)$ being constant, we expand the results of Ruiz's work in
\cite{R1} with a more general nonlinearity including the case of $f(u)$
being asymptotically linear at infinity, demonstrating the existence of
multiple positive solutions in Theorem \ref{t1}. In the presence of
non-constant $\rho(x)$, we make use of Theorem \ref{t1} to prove the
existence of ground state solutions in Theorem \ref{t2}; however, the
solution is shown to be uniquely zero in Theorem \ref{t2-1} when $\inf_{x\in
\mathbb{R}^{3}}\rho \left( x\right)$ exceeds a certain threshold value. In
Theorem {\ref{t3}, by considering system $(SP_{\rho })$ in a radially
symmetric setting, three positive solutions are obtained including a
non-radial ground state solution. }

Before presenting our main results in the theorems below, we need to state
the following maximization problems:
\begin{equation*}
\Lambda _{0}:=\sup_{u\in \mathbf{A}_{0}}\frac{\int_{\mathbb{R}^{3}}F\left(
u\right) dx-\frac{1}{2}\left\Vert u\right\Vert _{H^{1}}^{2}}{\int_{\mathbb{R}%
^{3}}\phi _{u}u^{2}dx}
\end{equation*}%
and%
\begin{equation*}
\overline{\Lambda }_{0}:=\sup_{u\in \overline{\mathbf{A}}_{0}}\frac{\int_{%
\mathbb{R}^{3}}f\left( u\right) udx-\left\Vert u\right\Vert _{H^{1}}^{2}}{%
\int_{\mathbb{R}^{3}}\phi _{u}u^{2}dx},
\end{equation*}%
where $\int_{\mathbb{R}^{3}}\phi _{u}u^{2}dx=\int_{\mathbb{R}^{3}}\rho
\left( x\right) \phi _{\rho ,u}u^{2}dx$ for $\rho \left( x\right) \equiv 1, $ and the sets
\begin{equation*}
\mathbf{A}_{0}:=\left\{ u\in H_{r}^{1}\left( \mathbb{R}^{3}\right) :\int_{%
\mathbb{R}^{3}}F\left( u\right) dx-\frac{1}{2}\left\Vert u\right\Vert
_{H^{1}}^{2}>0\right\}
\end{equation*}%
and%
\begin{equation*}
\overline{\mathbf{A}}_{0}:=\left\{ u\in H^{1}\left( \mathbb{R}^{3}\right)
:\int_{\mathbb{R}^{3}}f\left( u\right) udx-\left\Vert u\right\Vert
_{H^{1}}^{2}>0\right\} .
\end{equation*}%
Then there exist $M_{0},\overline{M}_{0}>0$ such that $0<\Lambda _{0}\leq
M_{0}$ and $0<\overline{\Lambda }_{0}\leq \overline{M}_{0}$ (see Appendix).
Our main results are obtained by considering the following equation,%
\begin{equation}
\begin{array}{ll}
-\Delta u+u+\lambda \phi _{u}u=f\left( u\right) & \text{ in }\mathbb{R}^{3}.%
\end{array}
\tag*{$\left( E_{\lambda }^{\infty }\right) $}
\end{equation}%
with its solutions given by the critical points of the energy functional $%
J_{\lambda }^{\infty }:H^{1}(\mathbb{R}^{3})\rightarrow \mathbb{R}$ defined
as%
\begin{equation*}
J_{\lambda }^{\infty }\left( u\right) =\frac{1}{2}\int_{\mathbb{R}%
^{3}}\left( |\nabla u|^{2}+u^{2}\right) dx+\frac{\lambda }{4}\int_{\mathbb{R}%
^{3}}\phi _{u}u^{2}dx-\int_{\mathbb{R}^{3}}F\left( u\right) dx.
\end{equation*}%
Then we have the following result for the autonomous case.

\begin{theorem}
\label{t1}Suppose that conditions$\ \left( F1\right) $ and $\left( F2\right)
$ hold. Then we have the following results.\newline
$\left( i\right) $ For every $0<\lambda <4\Lambda _{0},$ Equation $%
(E_{\lambda }^{\infty })$ has two positive radial solutions $v_{\lambda
}^{\left( 1\right) },v_{\lambda }^{\left( 2\right) }\in H_{r}^{1}\left(
\mathbb{R}^{3}\right) $ with $J_{\lambda }^{\infty }\left( v_{\lambda
}^{\left( 1\right) }\right) <0<J_{\lambda }^{\infty }\left( v_{\lambda
}^{\left( 2\right) }\right) .$\newline
$\left( ii\right) $ For every $\lambda >\overline{\Lambda }_{0},$ $u=0$ is
the unique solution of Equation $(E_{\lambda }^{\infty }).$
\end{theorem}

By conditions $\left( F1\right) $ and $\left( F2\right) ,$ there exist $%
q<p<3 $ and $C_{0}>0$ such that
\begin{equation}
f\left( s\right) \leq \frac{1}{4}s+C_{0}s^{p-1}\text{ for all }s\geq 0.
\label{1-8}
\end{equation}%
We now present our main results for the non-autonomous case.

\begin{theorem}
\label{t2}Suppose that conditions$\ \left( F1\right) ,\left( F2\right) $ and
$(D1)$ hold, and $0<\lambda <4\Lambda _{0}$ is given. In addition, we assume
that\newline
$\left( D2\right) $
\begin{equation}
0<\rho _{\min }:=\inf_{x\in \mathbb{R}^{3}}\rho \left( x\right)
<d_{0}:=\left( \frac{C_{0}2^{3\left( 4-p\right) /2}\left( 3-p\right) ^{3-p}}{%
p\left( p-2\right) ^{p-2}}\right) ^{1/\left( p-2\right) }<\rho _{\infty };
\label{1-9}
\end{equation}%
$\left( D3\right) $ $\int_{\mathbb{R}^{3}}\rho \left( x\right) \phi _{\rho
,v_{\lambda }^{\left( 1\right) }}\left( v_{\lambda }^{\left( 1\right)
}\right) ^{2}dx<\lambda \int_{\mathbb{R}^{3}}\phi _{v_{\lambda }^{\left(
1\right) }}\left( v_{\lambda }^{\left( 1\right) }\right) ^{2}dx,$ where $%
v_{\lambda }^{\left( 1\right) }$ is a radial positive solution of Equation $%
(E_{\lambda }^{\infty })$ as in Theorem \ref{t1}.\newline
Then Equation $(E_{\rho })$ has a positive ground state solution $u_{\rho }$
such that $J_{\rho }\left( u_{\rho }\right) <0.$
\end{theorem}

\begin{theorem}
\label{t2-1}Suppose that conditions$\ \left( F1\right) ,\left( F2\right) $
and $(D1)$ hold and $\rho _{\min }>\sqrt{\overline{\Lambda }_{0}}.$ Then $%
u=0 $ is the unique solution of Equation $(E_{\rho }).$
\end{theorem}

To study the symmetry breaking of ground state solutions, we consider the
following equation:%
\begin{equation}
\begin{array}{ll}
-\Delta u+u+\rho _{\varepsilon }\left( x\right) \phi _{\rho _{\varepsilon
},u}u=f\left( u\right) & \text{ in }\mathbb{R}^{3},%
\end{array}
\tag*{$\left( E_{\rho _{\varepsilon }}\right) $}
\end{equation}%
where $\rho _{\varepsilon }\left( x\right) =\rho \left( \varepsilon x\right)
$ and $\varepsilon >0.$ Then we have the following results.

\begin{theorem}
\label{t3}Suppose that conditions$\ \left( F1\right) -\left( F2\right) $ and
$(D1)$ hold, and $0<\lambda <4\Lambda _{0}$ is given. In addition, we assume
that\newline
$\left( D4\right) $ $\rho \left( x\right) =\rho \left( \left\vert
x\right\vert \right) ;$\newline
$\left( D5\right) $
\begin{equation*}
0<\rho \left( 0\right) <\min \left\{ d_{0},\sqrt{\lambda}\right\} \leq \max
\left\{ d_{0},\sqrt{\lambda }\right\} <\rho _{\infty }.
\end{equation*}%
\newline
Then Equation $(E_{\rho _{\varepsilon }})$ has three positive solutions $%
u_{\rho _{\varepsilon }}\in H^{1}\left( \mathbb{R}^{3}\right) $ and $v_{\rho
_{\varepsilon }}^{\left( 1\right) },v_{\rho _{\varepsilon }}^{\left(
2\right) }\in H_{r}^{1}\left( \mathbb{R}^{3}\right) $ such that
\begin{equation*}
J_{\rho _{\varepsilon }}\left( u_{\rho _{\varepsilon }}\right) <J_{\rho
_{\varepsilon }}\left( v_{\rho _{\varepsilon }}^{\left( 1\right) }\right)
<0<J_{\rho _{\varepsilon }}\left( v_{\rho _{\varepsilon }}^{\left( 2\right)
}\right) \text{ for }\varepsilon \text{ sufficiently small.}
\end{equation*}%
Furthermore, $u_{\rho _{\varepsilon }}$ is a non-radial ground state
solution of Equation $(E_{\rho _{\varepsilon }}).$
\end{theorem}

\begin{remark}
\label{R2}$\left( i\right) $ Suppose that conditions ${(D1)}$ and ${(D5)}$
hold. Let $v_{\lambda }^{\left( 1\right) }$ be a radial positive solution of
Equation $(E_{\lambda }^{\infty })$ as in Theorem \ref{t1} and let $\rho
\left( x_{0}\right) <\min \left\{ d_{0},\sqrt{\lambda }\right\} $ for some $%
x_{0}\in \mathbb{R}^{3}.$ Define $v_{\varepsilon }\left( x\right)
=v_{\lambda }^{\left( 1\right) }\left( x-\frac{x_{0}}{\varepsilon }\right) .$
Then it follows from condition $\left( D2\right) $ that%
\begin{eqnarray}
\int_{\mathbb{R}^{3}}\rho \left( x\right) \phi _{\rho ,v_{\varepsilon
}}v_{\varepsilon }^{2}dx &=&\int_{\mathbb{R}^{3}}\rho \left( \varepsilon
x+x_{0}\right) \phi _{\rho \left( \varepsilon x+x_{0}\right) ,v_{\lambda
}^{\left( 1\right) }}\left( v_{\lambda }^{\left( 1\right) }\right) ^{2}dx
\notag \\
&=&\rho ^{2}\left( x_{0}\right) \int_{\mathbb{R}^{3}}\phi _{v_{\lambda
}^{\left( 1\right) }}\left( v_{\lambda }^{\left( 1\right) }\right)
^{2}dx+o\left( \varepsilon \right)  \notag \\
&<&\lambda \int_{\mathbb{R}^{3}}\phi _{v_{\lambda }^{\left( 1\right)
}}\left( v_{\lambda }^{\left( 1\right) }\right) ^{2}dx\text{ for }%
\varepsilon \text{ sufficiently small.}  \notag
\end{eqnarray}%
This implies that when $\rho \left( x\right) $ is replaced by $\rho \left(
\varepsilon x+x_{0}\right) ,$ the condition $\left( D3\right) $ holds for $%
\varepsilon $ sufficiently small. Therefore, by Theorem \ref{t2}, Equation $%
(E_{\rho _{\varepsilon }})$ has a positive ground state solution $u_{\rho
_{\varepsilon }}\in H^{1}\left( \mathbb{R}^{3}\right) $ such that $J_{\rho
_{\varepsilon }}\left( u_{\rho _{\varepsilon }}\right) <0.$\newline
$\left( ii\right) $ Assume that the conditions hold in Theorem \ref{t3}.
Since $\rho \left( x\right) =\rho \left( \left\vert x\right\vert \right) $
and $\rho \left( 0\right) <\min \left\{ d_{0},\sqrt{\lambda }\right\} ,$
using an argument similar to that in part $\left( i\right) ,$ we have%
\begin{equation*}
\int_{\mathbb{R}^{3}}\rho \left( x\right) \phi _{\rho ,v_{\lambda }^{\left(
1\right) }}\left( v_{\lambda }^{\left( 1\right) }\right) ^{2}dx<\lambda
\int_{\mathbb{R}^{3}}\phi _{v_{\lambda }^{\left( 1\right) }}\left(
v_{\lambda }^{\left( 1\right) }\right) ^{2}dx\text{ for }\varepsilon \text{
sufficiently small,}
\end{equation*}%
since $\rho \left( 0\right) <\min \left\{ d_{0},\sqrt{\lambda }\right\} .$
This means that the symmetric case still holds in Theorem \ref{t2} implying
that Equation $(E_{\rho _{\varepsilon }})$ has a radial positive solution $%
v_{\rho _{\varepsilon }}\in H^{1}\left( \mathbb{R}^{3}\right) $ such that
\begin{equation*}
J_{\rho _{\varepsilon }}\left( v_{\rho _{\varepsilon }}\right) <0\text{ for }%
\varepsilon \text{ sufficiently small.}
\end{equation*}
\end{remark}

\begin{remark}
Under the assumption that $\rho \not\equiv 0,$ equation $\left( E_{\rho
}\right) $ can be regarded as a perturbation problem of the following
nonlinear Schr\"{o}dinger equation:%
\begin{equation}
\left\{
\begin{array}{ll}
-\Delta u+u=f\left( u\right) & \text{in }\mathbb{R}^{3}, \\
u\in H^{1}\left( \mathbb{R}^{3}\right) . &
\end{array}%
\right.  \tag{$E_{0}$}
\end{equation}%
It is known that equation $\left( E_{0}\right) $ has a positive (ground
state) solution under various conditions(c.f. Berestycki--Lions \cite{BeL},
Lions \cite{Li1,Li2}). The study of the uniqueness and symmetry of positive
solutions to equation $(E_{0})$ has a very long history, of particular
interest is the case where
\begin{equation*}
f\left( u\right) =\left\vert u\right\vert ^{q-2}u,\text{ for }2<q<2^{\ast },
\end{equation*}%
with $2^{\ast }=2N/\left( N-2\right) $ being the critical Sobolev exponent
in dimensions $N\geq 3$ and $2^{\ast }=\infty $ in dimensions $N=1,2.$ The
uniqueness of positive solutions was proven first by Coffman \cite{C} for $%
q=4$ and $N=3,$ and later by Kwong \cite{K} for the general case. The
symmetry of positive solutions was proven by Gidas et al.\cite{GNN1,GNN2}.
These results have subsequently been extended to include a larger class of
non-linearities by many authors; see for examples \cite%
{CEF,CL,J,KZ,L2,LR,Mc,MS,PS,PuS,ST}. Note that the solution loses its
uniqueness and symmetry properties when $\rho \not\equiv 0$ is assumed.
\end{remark}

This paper is organized as follows. In Section 2, we provide the necessary
preliminaries and prove that the energy functional $J_{\rho }$ is coercive
and bounded below in $H^{1}(\mathbb{R}^{3}).$ In Section 3, we show that $%
\inf_{u\in H^{1}(\mathbb{R}^{3})}J_{\rho }\left( u\right) <0.$ In Section 4,
we investigate the Palais-Smale condition of $J_{\rho }$ on $H^{1}(\mathbb{R}%
^{3}),$ which is followed by the proofs of Theorems \ref{t2} and \ref{t2-1}.
Finally, in Section 5 the proof of Theorem \ref{t3} is provided.

\section{Preliminaries}

We first establish the following estimates on the nonlinearity.

\begin{lemma}
\label{g3-2}Suppose that $2<p<3$ and $d>0.$ Let $C_{0}>0$ be as in $\left( %
\ref{1-8}\right) $ and $f_{d}\left( s\right) =\frac{1}{8}+\frac{d}{\sqrt{8}}%
s-\frac{C_{0}}{p}s^{p-2}$ for $s>0.$ Then there exist $d_{0}:=\left( \frac{%
C_{0}2^{3\left( 4-p\right) /2}\left( 3-p\right) ^{3-p}}{p\left( p-2\right)
^{p-2}}\right) ^{1/\left( p-2\right) }>0$ and $s_{0}\left( d\right) :=\left(
\frac{\sqrt{8}C_{0}\left( p-2\right) }{pd}\right) ^{1/\left( 3-p\right) }>0$
satisfying\newline
$\left( i\right) $ $f_{d}^{\prime }\left( s_{0}\left( d\right) \right) =0$
and $f_{d}\left( s_{0}\left( d_{0}\right) \right) =0;$\newline
$\left( ii\right) $ for each $d<d_{0}$ there exist $\eta _{d},\xi _{d}>0$
such that $\eta _{d}<s_{0}\left( d\right) <\xi _{d}$ and $f_{d}\left(
s\right) <0$ for all $s\in \left( \eta _{d},\xi _{d}\right) ;$\newline
$\left( iii\right) $ for each $d>d_{0},$ $f_{d}\left( s\right) >0$ for all $%
s>0.$
\end{lemma}

\begin{proof}
By a straightforward calculation, we can show that the results hold.
\end{proof}

Following the idea of Lions \cite{Lions} (or see \cite{R1}), we have%
\begin{eqnarray}
\frac{1}{\sqrt{8}}\int_{\mathbb{R}^{3}}\rho \left( x\right) \left\vert
u\right\vert ^{3}dx &=&\frac{1}{\sqrt{8}}\int_{\mathbb{R}^{3}}\left( -\Delta
\phi _{\rho ,u}\right) \left\vert u\right\vert dx=\frac{1}{\sqrt{8}}\int_{%
\mathbb{R}^{3}}\left\langle \nabla \phi _{\rho ,u},\nabla \left\vert
u\right\vert \right\rangle dx  \notag \\
&\leq &\frac{1}{4}\int_{\mathbb{R}^{3}}\left\vert \nabla u\right\vert ^{2}dx+%
\frac{1}{8}\int_{\mathbb{R}^{3}}\left\vert \nabla \phi _{\rho ,u}\right\vert
^{2}dx  \notag \\
&=&\frac{1}{4}\int_{\mathbb{R}^{3}}\left\vert \nabla u\right\vert ^{2}dx+%
\frac{1}{8}\int_{\mathbb{R}^{3}}\rho \left( x\right) \phi _{\rho ,u}u^{2}dx%
\text{ for all }u\in H^{1}(\mathbb{R}^{3}).  \label{2-5}
\end{eqnarray}%
Then by $\left( \ref{1-8}\right) $ and $\left( \ref{2-5}\right),$%
\begin{eqnarray}
J_{\rho }(u) &\geq &\frac{1}{2}\left\Vert u\right\Vert _{H^{1}}^{2}+\frac{1}{%
4}\int_{\mathbb{R}^{3}}\rho \left( x\right) \phi _{\rho ,u}u^{2}dx-\frac{%
C_{0}}{p}\int_{\mathbb{R}^{3}}\left\vert u\right\vert ^{p}dx-\frac{1}{8}%
\int_{\mathbb{R}^{3}}u^{2}dx  \notag \\
&\geq &\frac{1}{4}\left\Vert u\right\Vert _{H^{1}}^{2}+\int_{\mathbb{R}%
^{3}}u^{2}\left( \frac{1}{8}+\frac{1}{\sqrt{8}}\rho \left( x\right)
\left\vert u\right\vert -\frac{C_{0}}{p}\left\vert u\right\vert
^{p-2}\right) dx+\frac{1}{8}\int_{\mathbb{R}^{3}}\rho \left( x\right) \phi
_{\rho ,u}u^{2}dx  \notag \\
&=&\frac{1}{4}\left\Vert u\right\Vert _{H^{1}}^{2}+\frac{1}{2}\int_{\left\{
\rho (x)<d_{0}\right\} }u^{2}\left( \frac{1}{8}+\frac{1}{\sqrt{8}}\rho
\left( x\right) \left\vert u\right\vert -\frac{C_{0}}{p}\left\vert
u\right\vert ^{p-2}\right) dx  \notag \\
&&+\frac{1}{2}\int_{\left\{ \rho \left( x\right) \geq d_{0} \right\}
}u^{2}\left( \frac{1}{8}+\frac{1}{\sqrt{8}}\rho \left( x\right) \left\vert
u\right\vert -\frac{C_{0}}{p}\left\vert u\right\vert ^{p-2}\right) dx+\frac{1%
}{8}\int_{\mathbb{R}^{3}}\rho \left( x\right) \phi _{\rho ,u}u^{2}dx,
\label{2-6}
\end{eqnarray}%
where $C_{0}>0$ as in $\left( \ref{1-8}\right) .$ By Lemma \ref{g3-2} and $%
\left( \ref{2-6}\right) ,$ we have%
\begin{eqnarray}
J_{\rho }(u) &\geq &\frac{1}{4}\left\Vert u\right\Vert _{H^{1}}^{2}+\frac{1}{%
2}\int_{\left\{ \rho (x)<d_{0}\right\} }u^{2}\left( \frac{1}{8}+\frac{1}{%
\sqrt{8}}\rho \left( x\right) \left\vert u\right\vert -\frac{C_{0}}{p}%
\left\vert u\right\vert ^{p-2}\right) dx  \notag \\
&\geq &\frac{1}{4}\left\Vert u\right\Vert _{H^{1}}^{2}+\frac{1}{2}%
\int_{\left\{ \rho (x)<d_{0}\right\} }m_{\rho }\left( x\right) dx,
\label{2-7}
\end{eqnarray}%
where $m_{\rho }\left( x\right) =\inf_{s\geq 0}\left( \frac{1}{8}s^{2}+\frac{%
1}{\sqrt{8}}\rho \left( x\right) s^{3}-\frac{C_{0}}{p}s^{p}\right) <0$ for
all $x\in \left\{ \rho (x)<d_{0}\right\} .$ Note that%
\begin{equation*}
\inf_{x\in \left\{ \rho (x)<d_{0}\right\} }m_{\rho }\left( x\right) \leq
\frac{1}{8}s_{0}^{2}\left( \rho _{\min }\right) +\frac{\rho _{\min }}{\sqrt{2%
}}s_{0}^{3}\left( \rho _{\min }\right) -\frac{C_{0}}{p}s_{0}^{p}\left( \rho
_{\min }\right) <0,
\end{equation*}%
and%
\begin{equation}
0>\int_{\left\{ \rho (x)<d_{0}\right\} }m_{\rho }\left( x\right) dx\geq
\inf_{x\in \left\{ \rho (x)<d_{0}\right\} }m_{\rho }\left( x\right)
\left\vert \left\{ \rho (x)<d_{0}\right\} \right\vert ,  \label{2-8}
\end{equation}%
where $s_{0}\left( \rho _{\min }\right) =\left( \frac{\sqrt{8}C_{0}\left(
p-2\right) }{p\rho _{\min }}\right) ^{1/\left( 3-p\right) }.$ Furthermore,
the following statements are true.

\begin{theorem}
\label{g5}Suppose that conditions $\left( F1\right) ,\left( F2\right) ,(D1)$
and $\left( D2\right) $ hold. Then $J_{\rho }$ is coercive and bounded below
on $H^{1}(\mathbb{R}^{3}).$ Furthermore,
\begin{equation*}
\inf_{u\in H^{1}(\mathbb{R}^{3})}J_{\rho }(u)>\int_{\left\{ \rho
(x)<d_{0}\right\} }m_{\rho }\left( x\right) dx>-\infty .
\end{equation*}
\end{theorem}

\begin{proof}
By conditions $\left( D1\right) $ and $\left( D2\right) $, we can conclude
that
\begin{equation}
0<\left\vert \left\{ \rho (x)<d_{0}\right\} \right\vert <\infty .
\label{2-9}
\end{equation}%
Thus, by $\left( \ref{2-7}\right) -\left( \ref{2-9}\right) ,$%
\begin{equation*}
0>\int_{\left\{ \rho (x)<d_{0}\right\} }m_{\rho }\left( x\right) dx>-\infty
\end{equation*}%
and%
\begin{equation*}
J_{\rho }(u)\geq \frac{1}{4}\left\Vert u\right\Vert
_{H^{1}}^{2}+\int_{\left\{ \rho (x)<d_{0}\right\} }m_{\rho }\left( x\right)
dx.
\end{equation*}%
This completes the proof.
\end{proof}

\begin{lemma}
\label{g5-2}Suppose that conditions $\left( F1\right) ,\left( F2\right)
,(D1) $ and $\left( D2\right) $ hold. Let $u_{0}$ be a non-trivial solution
of the following equation:%
\begin{equation}
\begin{array}{ll}
-\Delta u+u+\rho _{\infty ,}\phi _{\rho _{\infty ,}u}u=f\left( u\right) &
\text{ in }\mathbb{R}^{3}.%
\end{array}
\tag*{$\left( E_{\rho _{\infty }}\right) $}
\end{equation}%
Then $J_{\rho _{\infty }}(u_{\lambda })>0,$ where $J_{\rho _{\infty
}}=J_{\rho }$ for $\rho \equiv \rho _{\infty }.$
\end{lemma}

\begin{proof}
By Lemma \ref{g3-2} and $\left( \ref{2-5}\right) -\left( \ref{2-6}\right) ,$
\begin{eqnarray*}
J_{\rho _{\infty }}(u) &\geq &\frac{1}{4}\left\Vert u\right\Vert
_{H^{1}}^{2}+\frac{1}{2}\int_{\mathbb{R}^{3}}u^{2}\left( \frac{1}{8}+\frac{1%
}{\sqrt{8}}\rho _{\infty }\left\vert u\right\vert -\frac{C_{1}}{p}\left\vert
u\right\vert ^{p-2}\right) dx \\
&>&0\text{ for all }u\in H^{1}(\mathbb{R}^{3})\setminus \left\{ 0\right\} .
\end{eqnarray*}%
This completes the proof.
\end{proof}

Next, we define the Palais--Smale (PS) sequences and (PS)--conditions in $%
H^{1}\left( \mathbb{R}^{3}\right) $ for $J_{\rho }$ as follows.

\begin{definition}
$(i)$ For $\beta \in \mathbb{R}\mathbf{,}$ a sequence $\left\{ u_{n}\right\}
$ is a $(PS)_{\beta }$--sequence in $H^{1}(\mathbb{R}^{3})$ for $J_{\rho }$
if $J_{\rho }(u_{n})=\beta +o(1)\;$and$\;J_{\rho }^{\prime }(u_{n})=o(1)\;$%
strongly in $H^{-1}\left( \mathbb{R}^{3}\right) $ as $n\rightarrow \infty .$%
\newline
$(ii)$ We say that $J_{\rho }$ satisfies the $(PS)_{\beta }$--condition in $%
H^{1}(\mathbb{R}^{3})$ if every (PS)$_{\beta }$--sequence in $H^{1}(\mathbb{R%
}^{3})$ for $J_{\rho }$ contains a convergent subsequence.
\end{definition}

\begin{proposition}
\label{l0}Suppose that conditions $\left( F1\right) ,\left( F2\right) $ and $%
(D1)$ hold. Let $\left\{ u_{n}\right\} $ be a bounded $(PS)_{\beta }$%
--sequence in $H^{1}(\mathbb{R}^{3})$ for $J_{\rho }.$ There exist a
subsequence $\left\{ u_{n}\right\} ,$ a number $m\in \mathbb{N}$, sequences $%
\left\{ x_{n}^{i}\right\} _{n=1}^{\infty }$ in $\mathbb{R}^{3}$, a function $%
u_{0}\in H^{1}(\mathbb{R}^{3}),$ and $0\not\equiv v^{i}\in H^{1}(\mathbb{R}%
^{3})$ when $1\leq i\leq m$ such that \newline
$(i)$ $|x_{n}^{i}|\rightarrow \infty \ $and $|x_{n}^{i}-x_{n}^{j}|%
\rightarrow \infty $ as $n\rightarrow \infty ,$ $1\leq i\neq j\leq m;$%
\newline
$(ii)$ $-\Delta u_{0}+u_{0}+\rho \left( x\right) \phi _{\rho
,u_{0}}u_{0}=f\left( u_{0}\right) \text{ in }\mathbb{R}^{3};$\newline
$(iii)$ $-\Delta v^{i}+v^{i}+\rho _{\infty }\phi _{\rho _{\infty
},v^{i}}v^{i}=f\left( v^{i}\right) \text{ in }\mathbb{R}^{3};$\newline
$(iv)$ $u_{n}=u_{0}+\underset{i=1}{\overset{m}{\sum }}v^{i}\left( \cdot
-x_{n}^{i}\right) +o(1)\;$strongly in $H^{1}(\mathbb{R}^{3});$ and\newline
$(v)$ $J_{\rho }(u_{n})=J_{\rho }(u_{0})+\underset{i=1}{\overset{m}{\sum }}%
J_{\rho _{\infty }}(v^{i})+o(1)$.
\end{proposition}

The proof follows the same argument of \cite[Lemma 4.1]{CV} or \cite[Lemma
5.1]{V} and is omitted here.

\begin{corollary}
\label{m2}Suppose that conditions $\left( F1\right) ,\left( F2\right) ,(D1)$
and $\left( D2\right) $ hold. Let $\left\{ u_{n}\right\} $ be a $(PS)_{\beta
}$--sequence in $H^{1}(\mathbb{R}^{3})$ for $J_{\rho }$ with $\beta <0.$
Then there exist a subsequence $\left\{ u_{n}\right\} $ and a nonzero $u_{0}$
in $H^{1}(\mathbb{R}^{3})$ such that $u_{n}\rightarrow u_{0}$ strongly in $%
H^{1}(\mathbb{R}^{3})$ and $J_{\rho }\left( u_{0}\right) =\beta .$
Furthermore, $u_{0}$ is a non-trivial solution of Equation $(E_{\rho }).$
\end{corollary}

\begin{proof}
Let $\left\{ u_{n}\right\} $ be a $(PS)_{\beta }$--sequence in $H^{1}(%
\mathbb{R}^{3})$ for $J_{\rho }$ with $\beta <0.$ By Theorem \ref{g5}, there
exist a subsequence $\left\{ u_{n}\right\} $ and $u_{0}\in H^{1}(\mathbb{R}%
^{3})$ such that $u_{n}\rightharpoonup u_{0}$ weakly in $H^{1}(\mathbb{R}%
^{3})$ and $J_{\rho }^{\prime }\left( u_{0}\right) =0.$ Moreover, by Lemma %
\ref{g5-2} and Proposition \ref{l0} $\left( iv\right) -\left( v\right) ,$ $%
u_{n}\rightarrow u_{0}$ strongly in $H^{1}(\mathbb{R}^{3})$ and $J_{\rho
}\left( u_{0}\right) =\beta .$ Thus, $u_{0}$ is a non-trivial solution of
Equation $(E_{\rho }).$ This completes the proof.
\end{proof}

\section{Proof of Theorem \protect\ref{t1}}

\textbf{We are now ready to prove Theorem \ref{t1}.}$\left( i\right) $ By
the definition of $\Lambda _{0}$, for each $\lambda <4\Lambda _{0}$ there
exists $v_{0}\in \mathbf{A}$ such that%
\begin{equation*}
\frac{\int_{\mathbb{R}^{3}}F\left( v_{0}\right) dx-\frac{1}{2}\left\Vert
v_{0}\right\Vert _{H^{1}}^{2}}{\int_{\mathbb{R}^{3}}\phi _{v_{0}}v_{0}^{2}dx}%
>\frac{\lambda }{4},
\end{equation*}%
this implies that%
\begin{equation}
J_{\lambda }^{\infty }\left( v_{0}\right) =\frac{1}{2}\left\Vert
v_{0}\right\Vert _{H^{1}}^{2}+\frac{\lambda }{4}\int_{\mathbb{R}^{3}}\phi
_{v_{0}}v_{0}^{2}dx-\int_{\mathbb{R}^{3}}F\left( v_{0}\right) dx<0.
\label{3-1}
\end{equation}%
Moreover, by $\left( \ref{1-8}\right) $, for each $u\in H_{r}^{1}\left(
\mathbb{R}^{3}\right) ,$
\begin{eqnarray*}
J_{\lambda }^{\infty }\left( u\right)  &=&\frac{1}{2}\left\Vert u\right\Vert
_{H^{1}}^{2}+\frac{\lambda }{4}\int_{\mathbb{R}^{3}}\phi _{u}u^{2}dx-\int_{%
\mathbb{R}^{3}}F\left( u\right) dx \\
&\geq &\frac{3}{8}\left\Vert u\right\Vert _{H^{1}}^{2}+\frac{\lambda }{4}%
\int_{\mathbb{R}^{3}}\phi _{u}u^{2}dx-\frac{C_{0}}{p}\int_{\mathbb{R}%
^{3}}\left\vert u\right\vert ^{p}dx,
\end{eqnarray*}%
where $2<p<3.$ Then by the same argument used by Ruiz in \cite[Theorem 4.3]%
{R1}, there exists $K>0$ such that
\begin{equation}
J_{\lambda }^{\infty }\left( u\right) \geq \frac{3}{8}\left\Vert
u\right\Vert _{H^{1}}^{2}+\frac{\lambda }{4}\int_{\mathbb{R}^{3}}\phi
_{u}u^{2}dx-\frac{C_{0}}{p}\int_{\mathbb{R}^{3}}\left\vert u\right\vert
^{p}dx\geq \frac{1}{8}\left\Vert u\right\Vert _{H^{1}}^{2}-K\text{ for all }%
u\in H_{r}^{1}\left( \mathbb{R}^{3}\right) ,  \label{3-4}
\end{equation}%
this implies that $J_{\lambda }$ is coercive and bounded below on $H_{r}^{1}(%
\mathbb{R}^{3})$. Using $\left( \ref{3-1}\right) $ and $\left( \ref{3-4}%
\right) $, we have%
\begin{equation}
-\infty <\alpha _{\lambda }^{\infty }:=\inf_{u\in H_{r}^{1}\left( \mathbb{R}%
^{3}\right) }J_{\lambda }^{\infty }(u)<0.  \label{25}
\end{equation}%
Then by the Ekeland variational principle (cf. \cite{E}) and Palais\
criticality principle (cf. \cite{P}), there exists a sequence $%
\{u_{n}\}\subset H_{r}^{1}(\mathbb{R}^{3})$ such that
\begin{equation*}
J_{\lambda }^{\infty }(u_{n})=\alpha _{\lambda }^{\infty }+o(1)\text{ and }%
\left( J_{\lambda }^{\infty }\right) ^{\prime }(u_{n})=o(1)\text{ in }%
H^{-1}\left( \mathbb{R}^{3}\right) .
\end{equation*}%
Again, adopting the argument used in \cite[Theorem 4.3]{R1}, there exists a
subsequence $\{u_{n}\}$ and $v_{\lambda }^{\left( 1\right) }\in
H_{r}^{1}\left( \mathbb{R}^{3}\right) \setminus \{0\}$ such that $%
u_{n}\rightarrow v_{\lambda }^{\left( 1\right) }$ strongly in $H^{1}\left(
\mathbb{R}^{3}\right) $ and $v_{\lambda }^{\left( 1\right) }$ is a solution
of Equation $(E_{\lambda }^{\infty }).$ This indicates that $J_{\lambda
}^{\infty }(v_{\lambda }^{\left( 1\right) })=\alpha _{\lambda }^{\infty }<0.$
By the maximum principle, we conclude that $v_{\lambda }^{\left( 1\right) }$
is a positive solution of Equation $(E_{\lambda }^{\infty }).$ Moreover, by
condition $\left( F1\right) $ and the Sobolev emdedding,%
\begin{eqnarray*}
J_{\lambda }^{\infty }(u) &\geq &\frac{3}{8}\left\Vert u\right\Vert
_{H^{1}}^{2}-\frac{C_{0}}{p}\int_{\mathbb{R}^{3}}\left\vert u\right\vert
^{p}dx \\
&\geq &\frac{3}{8}\left\Vert u\right\Vert _{H^{1}}^{2}-\frac{C_{0}}{%
pS_{p}^{p}}\left\Vert u\right\Vert _{H^{1}}^{p}\text{ for all }u\in
H^{1}\left( \mathbb{R}^{3}\right) .
\end{eqnarray*}%
This implies that there exist $\eta ,\kappa >0$ such that $\Vert v_{\lambda
}^{(1)}\Vert _{H^{1}}>\eta $ and%
\begin{equation*}
\max \{J_{\lambda }^{\infty }(0),J_{\lambda }^{\infty }(v_{\lambda
}^{(1)})\}<0<\kappa \leq \inf_{\Vert u\Vert _{H^{1}}=\eta }J_{\lambda
}^{\infty }(u).
\end{equation*}%
Define%
\begin{equation*}
\beta _{\lambda }^{\infty }=\inf_{\gamma \in \Gamma }\max_{0\leq \tau \leq
1}J_{\lambda }^{\infty }(\gamma (\tau )),
\end{equation*}%
where $\Gamma =\{\gamma \in C([0,1],H_{r}^{1}\left( \mathbb{R}^{3}\right)
):\gamma (0)=0,\gamma (1)=v_{\lambda }^{(1)}\}.$ Then by the mountain pass
theorem ( cf. \cite{E2,R}) and Palais\ criticality principle, there exists a
sequence $\{u_{n}\}\subset H_{r}^{1}\left( \mathbb{R}^{3}\right) $ such that
\begin{equation*}
J_{\lambda }^{\infty }(u_{n})\rightarrow \beta _{\lambda }^{\infty }\geq
\kappa \quad \text{and}\quad \Vert \left(
J_{\lambda }^{\infty }\right) ^{\prime }(u_{n})\Vert _{H^{-1}}\rightarrow
0,\quad \text{as}\ n\rightarrow \infty ,
\end{equation*}%
and using an argument similar to that in \cite[Theorem 4.3]{R1}, we have a
subsequence $\left\{ u_{n}\right\} $ and $v_{\lambda }^{\left( 2\right) }\in
H_{r}^{1}\left( \mathbb{R}^{3}\right) $ with $u_{n}\rightarrow v_{\lambda
}^{\left( 2\right) }$ strongly in $H^{1}\left( \mathbb{R}^{3}\right) $,
indicating that $J_{\lambda }^{\infty }\left( v_{\lambda }^{\left( 2\right)
}\right) =\beta _{\lambda }^{\infty }>0$ and $\left( J_{\lambda }^{\infty
}\right) ^{\prime }\left( v_{\lambda }^{\left( 2\right) }\right) =0$. Using
condition $\left( F1\right) ,$ we obtain that $v_{\lambda }^{\left( 2\right)
}$ is nonnegative on $\mathbb{R}^{3}.$ Applying the maximum principle then
gives the result that $v_{\lambda }^{\left( 2\right) }$ is a positive
solution of Equation $(E_{\lambda }^{\infty }).$

$\left( ii\right) $ Suppose that the contrary is true. Let $u_{0}$ be a
nontrivial solution of Equation $(E_{\lambda }^{\infty }).$ Then by the
definition of $\overline{\Lambda }_{0}$ and $\lambda >\overline{\Lambda }%
_{0},$
\begin{eqnarray*}
0 &=&\left\Vert u_{0}\right\Vert _{H^{1}}^{2}+\lambda \int_{\mathbb{R}%
^{3}}\phi _{u_{0}}u_{0}^{2}dx-\int_{\mathbb{R}^{3}}f\left( u_{0}\right)
u_{0}dx \\
&>&\left\Vert u_{0}\right\Vert _{H^{1}}^{2}+\overline{\Lambda }_{0}\int_{%
\mathbb{R}^{3}}\phi _{u_{0}}u_{0}^{2}dx-\int_{\mathbb{R}^{3}}f\left(
u_{0}\right) u_{0}dx\geq 0,
\end{eqnarray*}%
which is a contradiction. Hence we complete the proof.

\section{Proofs of Theorems \protect\ref{t2} and \protect\ref{t2-1}}

\begin{lemma}
\label{g4}Suppose that conditions $\left( F1\right) ,\left( F2\right) $ and $%
(D1)-\left( D3\right) $ hold. Then%
\begin{equation*}
-\infty <\alpha _{\rho }:=\inf_{u\in H^{1}\left( \mathbb{R}^{3}\right)
}J_{\rho }(u)<0.
\end{equation*}
\end{lemma}

\begin{proof}
Let $v_{\lambda }^{\left( 1\right) }$ be a radial positive solution of
Equation $(E_{\lambda }^{\infty })$ as in Theorem \ref{t1}. Applying
condition $\left( D3\right) $ gives%
\begin{eqnarray}
J_{\rho }\left( v_{\lambda }^{\left( 1\right) }\right) &=&\frac{1}{2}%
\left\Vert v_{\lambda }^{\left( 1\right) }\right\Vert _{H^{1}}^{2}+\frac{1}{4%
}\int_{\mathbb{R}^{3}}\rho \left( x\right) \phi _{\rho ,v_{\lambda }^{\left(
1\right) }}\left( v_{\lambda }^{\left( 1\right) }\right) ^{2}dx-\int_{%
\mathbb{R}^{3}}F\left( v_{\lambda }^{\left( 1\right) }\right) dx  \notag \\
&\leq &\frac{1}{2}\left\Vert v_{\lambda }^{\left( 1\right) }\right\Vert
_{H^{1}}^{2}+\frac{\lambda }{4}\int_{\mathbb{R}^{3}}\phi _{v_{\lambda
}^{\left( 1\right) }}\left( v_{\lambda }^{\left( 1\right) }\right)
^{2}dx-\int_{\mathbb{R}^{3}}F\left( v_{\lambda }^{\left( 1\right) }\right) dx
\notag \\
&<&0.  \label{4-1}
\end{eqnarray}%
Thus, by Theorem \ref{g5} and $\left( \ref{4-1}\right) ,$ we have
\begin{equation*}
-\infty <\alpha _{\rho }:=\inf_{u\in H^{1}(\mathbb{R}^{3})}J_{\rho }(u)<0.
\end{equation*}%
This completes the proof.
\end{proof}

\bigskip

\textbf{We are now ready to prove Theorem \ref{t2}.} As a consequence of
Theorem \ref{g5} and Lemma \ref{g4}, the functional $J_{\rho }$ is coercive
on $H^{1}(\mathbb{R}^{3})$ and
\begin{equation*}
-\infty <\alpha _{\rho }:=\inf_{u\in H^{1}(\mathbb{R}^{3})}J_{\rho }(u)<0.
\end{equation*}%
By the Ekeland variational principle, there exists a sequence $%
\{u_{n}\}\subset H^{1}(\mathbb{R}^{3})$ such that
\begin{equation*}
J_{\rho }(u_{n})=\alpha _{\rho }+o(1)\text{ and }J_{\rho }^{\prime
}(u_{n})=o(1)\text{ in }H^{-1}\left( \mathbb{R}^{3}\right) .
\end{equation*}%
Thus, by Theorem \ref{g5} and Corollary \ref{m2}, there exist a subsequence $%
\{u_{n}\}$ and $u_{\rho }\in H^{1}\left( \mathbb{R}^{3}\right) \setminus
\{0\}$ such that $u_{n}\rightarrow u_{\rho }$ strongly in $H^{1}\left(
\mathbb{R}^{3}\right) $ and $u_{\rho }$ is a solution of Equation $(E_{\rho
})$ indicating that $J_{\rho }(u_{\rho })=\alpha _{\rho }<0.$ Moreover,
condition $\left( F1\right) $ ensures that $u_{\rho }$ is nonnegative on $%
\mathbb{R}^{3}.$ Using the maximum principle, we may conclude that $u_{\rho
} $ is a positive ground state solution of Equation $(E_{\rho }).$

\bigskip

\textbf{We are now ready to prove Theorem \ref{t2-1}.} Suppose that the
contrary is true. Let $u_{0}$ be a nontrivial solution of Equation $(E_{\rho
}).$ Then by the definition of $\overline{\Lambda }_{0}$ and $\rho _{\min }>%
\sqrt{\overline{\Lambda }_{0}},$
\begin{eqnarray*}
0 &=&\left\Vert u_{0}\right\Vert _{H^{1}}^{2}+\int_{\mathbb{R}^{3}}\rho
\left( x\right) \phi _{\rho ,u_{0}}u_{0}^{2}dx-\int_{\mathbb{R}^{3}}f\left(
u_{0}\right) u_{0}dx \\
&\geq &\left\Vert u_{0}\right\Vert _{H^{1}}^{2}+\rho _{\min }^{2}\int_{%
\mathbb{R}^{3}}\phi _{u_{0}}u_{0}^{2}dx-\int_{\mathbb{R}^{3}}f\left(
u_{0}\right) u_{0}dx \\
&>&\left\Vert u_{0}\right\Vert _{H^{1}}^{2}+\overline{\Lambda }_{0}\int_{%
\mathbb{R}^{3}}\phi _{u_{0}}u_{0}^{2}dx-\int_{\mathbb{R}^{3}}f\left(
u_{0}\right) u_{0}dx\geq 0,
\end{eqnarray*}%
which is a contradiction. Hence we complete the proof.

\section{Proof of Theorem \protect\ref{t3}}

By conditions $\left( D1\right) $ and $\left( D5\right) ,$ without loss of
generality, we may assume that
\begin{equation*}
B^{3}\left( 0,1\right) \subset \mathrm{int}\left\{ x\in \mathbb{R}^{3}:\rho
\left( x\right) < \sqrt{\lambda}\right\} ,
\end{equation*}%
this implies that $B^{3}\left( 0,\frac{1}{\varepsilon }\right) \subset
\Omega _{\varepsilon }:=\mathrm{int}\left\{ x\in \mathbb{R}^{3}:\rho \left(
\varepsilon x\right) < \sqrt{\lambda}\right\} .$ As we know, $v_{\lambda
}^{\left( 1\right) }$ is a radial positive solution of Equation $\left(
E_{\lambda }^{\infty }\right) $ and $J_{\lambda }^{\infty }\left( v_{\lambda
}^{\left( 1\right) }\right) <0.$ For $R>0,$ we define a cut-off function $%
\psi _{R}\in C^{1}(\mathbb{R}^{3},\left[ 0,1\right] )$ as
\begin{equation*}
\psi _{R}\left( x\right) =\left\{
\begin{array}{ll}
1 & \left\vert x\right\vert <\frac{R}{2}, \\
0 & \left\vert x\right\vert >R,%
\end{array}%
\right.
\end{equation*}%
and $\left\vert \nabla \psi _{R}\right\vert \leq 1$ in $\mathbb{R}^{3}.$ Let
$u_{R}\left( x\right) =v_{\lambda }^{\left( 1\right) }\psi _{R}(x).$ Then it
is true that
\begin{equation}
\int_{\mathbb{R}^{3}}F\left( u_{R}\right) dx\rightarrow \int_{\mathbb{R}%
^{3}}F\left( v_{\lambda }^{\left( 1\right) }\right) dx\text{ as }%
R\rightarrow \infty ,  \label{6-5}
\end{equation}%
\begin{equation}
\left\Vert u_{R}\right\Vert _{H^{1}}\rightarrow \left\Vert v_{\lambda
}^{\left( 1\right) }\right\Vert _{H^{1}}\text{ as }R\rightarrow \infty ,
\label{6-6}
\end{equation}%
and%
\begin{equation}
\int_{\mathbb{R}^{3}}\phi _{u_{R}}u_{R}^{2}dx\rightarrow \int_{\mathbb{R}%
^{3}}\phi _{v_{\lambda }^{\left( 1\right) }}\left( v_{\lambda }^{\left(
1\right) }\right) ^{2}dx\text{ as }R\rightarrow \infty ,  \label{6-9}
\end{equation}%
where $\int_{\mathbb{R}^{3}}\phi _{u}u^{2}dx=\int_{\mathbb{R}^{3}}\rho
\left( x\right) \phi _{\rho ,u}u^{2}dx$ for $\rho \equiv 1.$ Since $%
J_{\lambda }^{\infty }\in C^{1}\left( H^{1}\left( \mathbb{R}^{3}\right) ,%
\mathbb{R}\right) $ and $J_{\lambda }^{\infty }\left( v_{\lambda }^{\left(
1\right) }\right) <0$, by $\left( \ref{6-5}\right) -\left( \ref{6-9}\right)
, $ there exists $R_{0}>0$ such that
\begin{equation}
J_{\lambda }^{\infty }\left( u_{R_{0}}\right) <0.  \label{6-10}
\end{equation}%
Let%
\begin{equation*}
u_{R_{0},N}^{\left( i\right) }\left( x\right) =v_{\lambda }\left(
x+iN^{3}e\right) \psi _{R_{0}}\left( x+iN^{3}e\right)
\end{equation*}%
for $e\in \mathbb{S}^{2}$ and $i=1,2,\ldots ,N$, where $N^{3}>2R_{0}.$ Let $%
0<\varepsilon _{N}\leq \frac{1}{N^{4}+R_{0}}.$ Then we have the following
result,%
\begin{equation*}
\text{\textrm{supp}}u_{R_{0},N}^{\left( i\right) }\left( x\right) \subset
B^{3}\left( 0,\frac{1}{\varepsilon _{N}}\right) \text{ for all }i=1,2,\ldots
,N.
\end{equation*}%
Clearly, $\varepsilon _{N}\rightarrow 0^{+}$ as $N\rightarrow \infty .$
Moreover, using condition $\left( D1\right) $, we deduce that
\begin{eqnarray*}
\left\Vert u_{R_{0},N}^{\left( i\right) }\right\Vert _{H^{1}}^{2}
&=&\left\Vert u_{R_{0}}\right\Vert _{H^{1}}^{2}\text{ for all }N, \\
\int_{\mathbb{R}^{3}}F\left( u_{R_{0},N}^{\left( i\right) }\right) dx
&=&\int_{\mathbb{R}^{3}}F\left( u_{R_{0}}\right) dx\text{ for all }N,
\end{eqnarray*}%
and
\begin{eqnarray*}
&&\int_{\mathbb{R}^{3}}\rho \left( \varepsilon _{N}x\right) \phi _{\rho
_{\varepsilon _{N}},u_{R_{0},N}^{\left( i\right) }}\left[ u_{R_{0},N}^{%
\left( i\right) }\right] ^{2}dx \\
&=&\int_{\mathbb{R}^{3}}\int_{\mathbb{R}^{3}}\rho \left( \varepsilon
_{N}x\right) \rho \left( \varepsilon _{N}y\right) \frac{\left[
u_{R_{0},N}^{\left( i\right) }\left( x\right) \right] ^{2}\left[
u_{R_{0},N}^{\left( i\right) }\left( y\right) \right] ^{2}}{4\pi \left\vert
x-y\right\vert }dxdy \\
&=&\int_{\mathbb{R}^{3}}\int_{\mathbb{R}^{3}}\rho \left( \varepsilon
_{N}x-i\varepsilon _{N}N^{3}e\right) \rho \left( \varepsilon
_{N}y-i\varepsilon _{N}N^{3}e\right) \frac{\left[ u_{R}\left( x\right) %
\right] ^{2}\left[ u_{R}\left( y\right) \right] ^{2}}{4\pi \left\vert
x-y\right\vert }dxdy.
\end{eqnarray*}%
Since $0<\varepsilon _{N}\leq \frac{1}{N^{4}+R_{0}}$, there exists $N_{0}>0$
with $N_{0}^{3}>2R_{0}$ such that for every $N\geq N_{0},$ we have%
\begin{eqnarray*}
\int_{\mathbb{R}^{3}}\rho \left( \varepsilon _{N}x\right) \phi _{\rho
_{\varepsilon _{N}},u_{R_{0},N}^{\left( i\right) }}\left[ u_{R_{0},N}^{%
\left( i\right) }\right] ^{2}dx &=&\int_{\mathbb{R}^{3}}\int_{\mathbb{R}%
^{3}}\rho \left( \varepsilon _{N}x\right) \rho \left( \varepsilon
_{N}y\right) \frac{\left[ u_{R_{0},N}^{\left( i\right) }\left( x\right) %
\right] ^{2}\left[ u_{R_{0},N}^{\left( i\right) }\left( y\right) \right] ^{2}%
}{4\pi \left\vert x-y\right\vert }dxdy \\
&< &\lambda \int_{\mathbb{R}^{3}}\int_{\mathbb{R}^{3}}\frac{%
u_{R_{0}}^{2}\left( x\right) u_{R_{0}}^{2}\left( y\right) }{4\pi \left\vert
x-y\right\vert }dxdy \\
&=&\lambda \int_{\mathbb{R}^{3}}\phi _{u_{R_{0},N}^{\left( i\right) }}\left[
u_{R_{0},N}^{\left( i\right) }\right] ^{2}dx,
\end{eqnarray*}%
for all $e\in \mathbb{S}^{2}$ and $i=1,2,\ldots ,N.$ Let
\begin{equation*}
w_{R_{0},N}\left( x\right) =\sum_{i=1}^{N}u_{R_{0},N}^{\left( i\right) }.
\end{equation*}%
Observe that $w_{R_{0},N}$ is a sum of translation of $u_{R_{0}}.$ When $%
N^{3}\geq N_{0}^{3}>2R_{0}$, the summands have disjoint support and%
\begin{equation}
\text{\textrm{supp}}w_{R_{0},N}\left( x\right) \subset B^{3}\left( 0,\frac{1%
}{\varepsilon _{N}}\right) .  \label{13}
\end{equation}%
In this case we have,%
\begin{equation}
\left\Vert w_{R_{0},N}\right\Vert _{H^{1}}^{2}=N\Vert u_{R_{0}}\Vert
_{H^{1}}^{2},  \label{14}
\end{equation}%
\begin{equation}
\int_{\mathbb{R}^{3}}F\left( w_{R_{0},N}\right) dx=\sum_{i=1}^{N}\int_{%
\mathbb{R}^{3}}F\left( u_{R_{0},N}^{\left( i\right) }\right) dx=N\int_{%
\mathbb{R}^{3}}F\left( u_{R_{0}}\right) dx  \label{15}
\end{equation}%
and%
\begin{eqnarray*}
\int_{\mathbb{R}^{3}}\rho \left( \varepsilon _{N}x\right) \phi _{\rho
_{\varepsilon }w_{R,N}}w_{R,N}^{2}dx &<&\lambda \int_{\mathbb{R}^{3}}\phi
_{w_{R,N}}w_{R,N}^{2}dx.
\end{eqnarray*}%
Moreover,%
\begin{eqnarray}
\int_{\mathbb{R}^{3}}\phi _{w_{R,N}}w_{R,N}^{2}dx &=&\int_{\mathbb{R}%
^{3}}\int_{\mathbb{R}^{3}}\frac{w_{R_{0},N}^{2}\left( x\right)
w_{R_{0},N}^{2}\left( y\right) }{4\pi \left\vert x-y\right\vert }dxdy  \notag
\\
&=&\sum_{i=1}^{N}\int_{\mathbb{R}^{3}}\int_{\mathbb{R}^{3}}\frac{\left[
u_{R_{0},N}^{\left( i\right) }\left( x\right) \right] ^{2}\left[
u_{R_{0},N}^{\left( i\right) }\left( y\right) \right] ^{2}}{4\pi \left\vert
x-y\right\vert }dxdy  \notag \\
&&+\sum_{i\neq j}^{N}\int_{\mathbb{R}^{3}}\int_{\mathbb{R}^{3}}\frac{\left[
u_{R_{0},N}^{\left( i\right) }\left( x\right) \right] ^{2}\left[
u_{R_{0},N}^{\left( j\right) }\left( y\right) \right] ^{2}}{4\pi \left\vert
x-y\right\vert }dxdy  \notag \\
&=&N\int_{\mathbb{R}^{3}}\int_{\mathbb{R}^{3}}\frac{u_{R_{0}}^{2}\left(
x\right) u_{R_{0}}^{2}\left( y\right) }{4\pi \left\vert x-y\right\vert }dxdy
\notag \\
&&+\sum_{i\neq j}^{N}\int_{\mathbb{R}^{3}}\int_{\mathbb{R}^{3}}\frac{\left[
u_{R_{0},N}^{\left( i\right) }\left( x\right) \right] ^{2}\left[
u_{R_{0},N}^{\left( j\right) }\left( y\right) \right] ^{2}}{4\pi \left\vert
x-y\right\vert }dxdy.  \label{16}
\end{eqnarray}%
A straightforward calculation gives
\begin{equation*}
0<\sum_{i\neq j}^{N}\int_{\mathbb{R}^{3}}\int_{\mathbb{R}^{3}}\frac{\left[
u_{R_{0},N}^{\left( i\right) }\left( x\right) \right] ^{2}\left[
u_{R_{0},N}^{\left( j\right) }\left( y\right) \right] ^{2}}{4\pi \left\vert
x-y\right\vert }dxdy\leq \frac{N^{2}-N}{N^{3}-2R_{0}}\left( \int_{\mathbb{R}%
^{3}}v_{\lambda }^{2}\left( x\right) dx\right) ^{2},
\end{equation*}%
implying that%
\begin{equation}
\sum_{i\neq j}^{N}\int_{\mathbb{R}^{3}}\int_{\mathbb{R}^{3}}\frac{\left[
u_{R_{0},N}^{\left( i\right) }\left( x\right) \right] ^{2}\left[
u_{R_{0},N}^{\left( j\right) }\left( y\right) \right] ^{2}}{4\pi \left\vert
x-y\right\vert }dxdy\rightarrow 0\text{ as }N\rightarrow \infty .  \label{17}
\end{equation}%
We can now adopt the idea of multibump technique by Ruiz \cite{R1} (also see
\cite{MMV}) and the following results are obtained.

\begin{lemma}
\label{m4}Suppose that conditions $\left( F1\right) ,\left( F2\right)
,\left( D1\right) ,\left( D4\right) $ and $\left( D5\right) $ hold. Then
\begin{equation}
\alpha _{\rho _{\varepsilon }}\rightarrow -\infty \text{ as }\varepsilon
\rightarrow 0^{+}.  \label{eqq56}
\end{equation}
\end{lemma}

\begin{proof}
By $\left( \ref{6-10}\right) -\left( \ref{17}\right) ,$ we obtain
\begin{eqnarray*}
&&J_{\rho _{\varepsilon _{N}}}\left( w_{R,N}\right) =\frac{1}{2}\left\Vert
w_{R_{0},N}\right\Vert _{H^{1}}^{2}+\frac{1}{4}\int_{\mathbb{R}^{3}}\rho
\left( \varepsilon _{N}x\right) \phi _{\rho _{\varepsilon
}w_{R,N}}w_{R,N}^{2}dx-\int_{\mathbb{R}^{3}}F\left( w_{R_{0},N}\right) dx \\
&\leq &\frac{N}{2}\Vert u_{R_{0}}\Vert -N\int_{\mathbb{R}^{3}}F\left(
u_{R_{0}}\right) dx \\
&&+\frac{\lambda N}{4}\int_{\mathbb{R}^{3}}\int_{\mathbb{R}^{3}}\frac{%
u_{R_{0}}^{2}\left( x\right) u_{R_{0}}^{2}\left( y\right) }{4\pi \left\vert
x-y\right\vert }dxdy+\frac{\lambda }{4}\sum_{i\neq j}^{N}\int_{\mathbb{R}%
^{3}}\int_{\mathbb{R}^{3}}\frac{\left[ u_{R_{0},N}^{\left( i\right) }\left(
x\right) \right] ^{2}\left[ u_{R_{0},N}^{\left( j\right) }\left( y\right) %
\right] ^{2}}{4\pi \left\vert x-y\right\vert }dxdy \\
&\leq &NJ_{\lambda }^{\infty }\left( u_{R_{0}}\right) +C_{0}\text{ for some }%
C_{0}>0
\end{eqnarray*}%
and
\begin{equation*}
J_{\rho _{\varepsilon _{N}}}\left( w_{R,N}\right) \rightarrow -\infty \text{
as }N\rightarrow \infty.
\end{equation*}%
Thus we arrive at (\ref{eqq56}).
\end{proof}

\begin{lemma}
\label{m5}Suppose that conditions $\left( F1\right) ,\left( F2\right) ,{(D1)}%
,{(D4)}$ and $\left( D5\right) $ hold. Then there exists $M>0$ independent
of $\varepsilon $ such that $0>\inf_{u\in H_{r}^{1}}J_{\rho _{\varepsilon
}}\left( u\right) \geq -M$ for $\varepsilon $ sufficiently small.
\end{lemma}

\begin{proof}
Since $\rho \left( x\right) =\rho \left( \left\vert x\right\vert \right) ,$
by Remark \ref{R2} $\left( i\right) $,%
\begin{eqnarray*}
\inf_{u\in H_{r}^{1}\left( \mathbb{R}^{3}\right) }J_{\rho _{\varepsilon
}}\left( u\right) &\leq &\frac{1}{2}\left\Vert v_{\lambda }^{\left( 1\right)
}\right\Vert _{H^{1}}^{2}+\frac{\lambda }{4}\int_{\mathbb{R}^{3}}\rho \left(
\varepsilon x\right) \phi _{\rho _{\varepsilon },v_{\lambda }^{\left(
1\right) }}\left( v_{\lambda }^{\left( 1\right) }\right) ^{2}dx-\int_{%
\mathbb{R}^{3}}F\left( v_{\lambda }^{\left( 1\right) }\right) dx \\
&\leq &\frac{1}{2}\left\Vert v_{\lambda }^{\left( 1\right) }\right\Vert
_{H^{1}}^{2}+\frac{\lambda }{4}\int_{\mathbb{R}^{3}}\phi _{v_{\lambda
}^{\left( 1\right) }}v_{\lambda }^{2}dx-\int_{\mathbb{R}^{3}}F\left(
v_{\lambda }^{\left( 1\right) }\right) dx \\
&<&0\text{ for }\varepsilon \text{ sufficiently small.}
\end{eqnarray*}%
Moreover, by $\left( \ref{1-8}\right) $ and applying the argument in Ruiz
\cite[Theorem 4.3]{R1}, there exists $M>0$ such that%
\begin{eqnarray*}
J_{\rho _{\varepsilon }}\left( u\right) &\geq &\frac{1}{2}\left\Vert
u\right\Vert _{H^{1}}^{2}+\frac{\rho _{\min }^{2}}{4}\int_{\mathbb{R}%
^{3}}\phi _{u}u^{2}dx-\int_{\mathbb{R}^{3}}F\left( u\right) dx \\
&\geq &\frac{1}{2}\left\Vert u\right\Vert _{H^{1}}^{2}+\frac{\rho _{\min
}^{2}}{4}\int_{\mathbb{R}^{3}}\phi _{u}u^{2}dx-\frac{1}{8}\int_{\mathbb{R}%
^{3}}u^{2}dx-\frac{C_{1}}{p}\int_{\mathbb{R}^{3}}\left\vert u\right\vert
^{p}dx \\
&\geq &\frac{3}{8}\left\Vert u\right\Vert _{H^{1}}^{2}+\frac{\rho _{\min
}^{2}}{4}\int_{\mathbb{R}^{3}}\phi _{u}u^{2}dx-\frac{C_{1}}{p}\int_{\mathbb{R%
}^{3}}\left\vert u\right\vert ^{p}dx>-M\text{ for all }u\in H_{r}^{1}\left(
\mathbb{R}^{3}\right) ,
\end{eqnarray*}%
and so $\inf_{u\in H_{r}^{1}\left( \mathbb{R}^{3}\right) }J_{\rho
_{\varepsilon }}\left( u\right) \geq -M.$ This completes the proof.
\end{proof}

Define
\begin{equation*}
\theta _{\rho _{\varepsilon }}:=\inf_{u\in H_{r}^{1}\left( \mathbb{R}%
^{3}\right) }J_{\rho _{\varepsilon }}\left( u\right) .
\end{equation*}%
By Lemmas \ref{m4} and \ref{m5}, we have%
\begin{equation}
\alpha _{\rho _{\varepsilon }}<\theta _{\rho _{\varepsilon }}<0\text{ for }%
\varepsilon >0\text{ sufficiently small.}  \label{5-2}
\end{equation}%
Then by the Ekeland variational principle and Palais\ criticality principle,
for $\varepsilon $ small enough, there exists a sequence $\{u_{n}\}\subset
H_{r}^{1}\left( \mathbb{R}^{3}\right) $ such that%
\begin{equation}
J_{\rho _{\varepsilon }}(u_{n})=\theta _{\rho _{\varepsilon }}+o(1)\text{
and }J_{_{\rho _{\varepsilon }}}^{\prime }(u_{n})=o(1)\text{ in }%
H^{-1}\left( \mathbb{R}^{3}\right) .  \label{5-3}
\end{equation}

\bigskip

\textbf{We are now ready to prove Theorem \ref{t3}.} Given that \textbf{\ }$%
\{u_{n}\}\subset H_{r}^{1}\left( \mathbb{R}^{3}\right) $ satisfies
\begin{equation*}
J_{\rho _{\varepsilon }}(u_{n})=\theta _{\rho _{\varepsilon }}+o(1)\text{
and }J_{_{\rho _{\varepsilon }}}^{\prime }(u_{n})=o(1)\text{ in }%
H^{-1}\left( \mathbb{R}^{3}\right) .
\end{equation*}%
Then Theorem \ref{g5} ensures that $\{u_{n}\}$ is bounded. Without loss of
generality, we can assume that there exists $v_{\rho _{\varepsilon }}\in
H_{r}^{1}\left( \mathbb{R}^{3}\right) $ such that $u_{n}\rightharpoonup
v_{\rho _{\varepsilon }}$ weakly in $H^{1}\left( \mathbb{R}^{3}\right) .$
Moreover, by Ruiz \cite[Lemma 2.1]{R1}, $J_{_{\rho _{\varepsilon }}}^{\prime
}(v_{\lambda ,\varepsilon })=0$ in $H^{-1}\left( \mathbb{R}^{3}\right) $ and
$u_{n}\rightarrow v_{\rho _{\varepsilon }}$ strongly in $H^{1}\left( \mathbb{%
R}^{3}\right) ,$ implying that $J_{\rho _{\varepsilon }}(v_{\rho
_{\varepsilon }})=\theta _{\rho _{\varepsilon }}.$ Thus, $v_{\rho
_{\varepsilon }}$ is a radial ground state solution of Equation $(E_{\rho
_{\varepsilon }}).$ Using condition $\left( F1\right) ,$ we have the result
that $v_{\rho _{\varepsilon }}$ is nonnegative on $\mathbb{R}^{3}$ and
applying the maximum principle, we conclude that $v_{\rho _{\varepsilon }}$
is a positive solution of Equation $(E_{\rho }).$ Therefore, by Theorem \ref%
{t2} and $\left( \ref{5-2}\right) ,$ Equation $(E_{\rho _{\varepsilon }})$
has two positive solutions $u_{\rho _{\varepsilon }}\in H^{1}\left( \mathbb{R%
}^{3}\right) $ and $v_{\rho _{\varepsilon }}\in H_{r}^{1}\left( \mathbb{R}%
^{3}\right) $ such that
\begin{equation*}
\alpha _{\rho _{\varepsilon }}=J_{\rho _{\varepsilon }}\left( u_{\rho
_{\varepsilon }}\right) <\theta _{\rho _{\varepsilon }}=J_{\rho
_{\varepsilon }}\left( v_{\rho _{\varepsilon }}\right) <0
\end{equation*}%
for $\varepsilon $ sufficiently small. Since
\begin{equation*}
\alpha _{\rho _{\varepsilon }}=\inf_{u\in H^{1}\left( \mathbb{R}^{3}\right)
}J_{\rho _{\varepsilon }}\left( u\right) <\theta _{\rho _{\varepsilon
}}=\inf_{u\in H_{r}^{1}\left( \mathbb{R}^{3}\right) }J_{\rho _{\varepsilon
}}\left( u\right) \text{ for }\varepsilon \text{ sufficiently small}
\end{equation*}%
and $v_{\rho _{\varepsilon }}$ is a radial ground state solution of Equation
$(E_{\rho _{\varepsilon }}),$ we can conclude that $u_{\rho _{\varepsilon }}$
is a non-radial ground state solution of Equation $(E_{\rho _{\varepsilon
}}).$ On the other hand, we have an isolated minimum at $0$ and also an
absolute minimum $u_{\rho _{\varepsilon }}.$ Moreover, by condition $\left(
F1\right) $ and the Sobolev emdedding,%
\begin{eqnarray*}
J_{\rho _{\varepsilon }}(u) &\geq &\frac{3}{8}\left\Vert u\right\Vert
_{H^{1}}^{2}-\frac{C_{1}}{p}\int_{\mathbb{R}^{3}}\left\vert u\right\vert
^{p}dx \\
&\geq &\frac{3}{8}\left\Vert u\right\Vert _{H^{1}}^{2}-\frac{C_{1}}{%
pS_{p}^{p}}\left\Vert u\right\Vert _{H^{1}}^{p}\text{ for all }u\in
H^{1}\left( \mathbb{R}^{3}\right) .
\end{eqnarray*}%
This implies that there exist $\eta ,\kappa >0$ such that $\left\Vert
v_{\rho _{\varepsilon }}\right\Vert _{H^{1}}^{2}>\eta $ and
\begin{equation*}
\max \{J_{\rho _{\varepsilon }}(0),J_{\rho _{\varepsilon }}(v_{\rho
_{\varepsilon }})\}<0<\kappa \leq \inf_{\Vert u\Vert _{H^{1}}=\eta }J_{\rho
_{\varepsilon }}(u)\text{ for }\varepsilon \text{ sufficiently small.}
\end{equation*}%
Define%
\begin{equation*}
\beta _{\rho _{\varepsilon }}=\inf_{\gamma \in \Gamma }\max_{0\leq \tau \leq
1}J_{\rho _{\varepsilon }}(\gamma (\tau )),
\end{equation*}%
where $\Gamma =\{\gamma \in C([0,1],H_{r}^{1}\left( \mathbb{R}^{3}\right)
):\gamma (0)=0,\gamma (1)=v_{\rho _{\varepsilon }}\}.$ Then by the mountain
pass theorem ( cf. \cite{E2,R}) and Palais\ criticality principle, there
exists a sequence $\{u_{n}\}\subset H_{r}^{1}\left( \mathbb{R}^{3}\right) $
such that
\begin{equation*}
J_{\rho _{\varepsilon }}(u_{n})\rightarrow \beta _{\rho _{\varepsilon }}\geq
\kappa \quad \text{and}\quad \Vert J_{\rho
_{\varepsilon }}^{\prime }(u_{n})\Vert _{H^{-1}}\rightarrow 0,\quad \text{as}%
\ n\rightarrow \infty .
\end{equation*}%
Adopting the argument used in \cite[Theorem 4.3]{R1}, we have a subsequence $%
\left\{ u_{n}\right\} $ and $\widehat{v}_{\rho _{\varepsilon }}\in
H_{r}^{1}\left( \mathbb{R}^{3}\right) $ with $u_{n}\rightarrow \widehat{v}%
_{\rho _{\varepsilon }}$ strongly in $H^{1}\left( \mathbb{R}^{3}\right) $,
implying that $J_{\rho _{\varepsilon }}\left( \widehat{v}_{\rho
_{\varepsilon }}\right) =\beta _{\rho _{\varepsilon }}>0$ and $J_{\rho
_{\varepsilon }}^{\prime }\left( \widehat{v}_{\rho _{\varepsilon }}\right)
=0 $. Condition $\left( F1\right) $ ensures that $\widehat{v}_{\rho
_{\varepsilon }}$ is nonnegative on $\mathbb{R}^{3}$ and the maximum
principle gives the result that $\widehat{v}_{\rho _{\varepsilon }}$ is a
positive solution of Equation $(E_{\rho _{\varepsilon }}).$ Therefore, we
conclude that Equation $(E_{\rho _{\varepsilon }})$ has three positive
solutions $u_{\rho _{\varepsilon }}\in H^{1}\left( \mathbb{R}^{3}\right) $
and $\widehat{v}_{\rho _{\varepsilon }},v_{\rho _{\varepsilon }}\in
H_{r}^{1}\left( \mathbb{R}^{3}\right) $ such that
\begin{equation*}
\alpha _{\rho _{\varepsilon }}=J_{\rho _{\varepsilon }}\left( u_{\rho
_{\varepsilon }}\right) <\theta _{\rho _{\varepsilon }}=J_{\rho
_{\varepsilon }}\left( v_{\rho _{\varepsilon }}\right) <0<\beta _{\rho
_{\varepsilon }}=J_{\rho _{\varepsilon }}\left( \widehat{v}_{\rho
_{\varepsilon }}\right) .
\end{equation*}%
This completes the proof.

\section{Appendix}

Let%
\begin{eqnarray*}
\mathbf{A}_{0} &:&=\left\{ u\in H_{r}^{1}\left( \mathbb{R}^{3}\right) :\int_{%
\mathbb{R}^{3}}F\left( u\right) dx-\frac{1}{2}\left\Vert u\right\Vert
_{H^{1}}^{2}>0\right\} ; \\
\overline{\mathbf{A}}_{0} &:&=\left\{ u\in H^{1}\left( \mathbb{R}^{3}\right)
:\int_{\mathbb{R}^{3}}f\left( u\right) udx-\left\Vert u\right\Vert
_{H^{1}}^{2}>0\right\}
\end{eqnarray*}%
and
\begin{eqnarray*}
\Lambda _{0} &:&=\sup_{u\in \mathbf{A}_{0}}\frac{\int_{\mathbb{R}%
^{3}}F\left( u\right) dx-\frac{1}{2}\left\Vert u\right\Vert _{H^{1}}^{2}}{%
\int_{\mathbb{R}^{3}}\phi _{u}u^{2}dx}; \\
\overline{\Lambda }_{0} &:&=\sup_{u\in \overline{\mathbf{A}}_{0}}\frac{\int_{%
\mathbb{R}^{3}}f\left( u\right) udx-\left\Vert u\right\Vert _{H^{1}}^{2}}{%
\int_{\mathbb{R}^{3}}\phi _{u}u^{2}dx}.
\end{eqnarray*}%
Then we have the following results.

\begin{theorem}
\label{t6-1}Suppose that conditions $\left( F1\right) $ and $\left(
F2\right) $ hold. Then we have\newline
$\left( i\right) $ $\mathbf{A}_{0}$ is a nonempty set.\newline
$\left( ii\right) $ There exists $M_{0}>0$ such that $0<\Lambda _{0}\leq
M_{0}.$
\end{theorem}

\begin{proof}
$\left( i\right) $ If $q=2,$ then by $a_{q}>1$ and Fatou's lemma, for $u\in
H_{r}^{1}\left( \mathbb{R}^{3}\right) $ with $\left\Vert u\right\Vert
_{H^{1}}^{2}-a_{q}\int_{\mathbb{R}^{3}}u^{2}dx<0,$ we have%
\begin{equation*}
\lim_{t\rightarrow \infty }\frac{1}{t^{2}}\left[ \frac{1}{2}\left\Vert
tu\right\Vert _{H^{1}}^{2}-\int_{\mathbb{R}^{3}}F\left( tu\right) dx\right]
=2\left( \left\Vert u\right\Vert _{H^{1}}^{2}-a_{q}\int_{\mathbb{R}%
^{3}}u^{2}dx\right) <0,
\end{equation*}%
and so there exists $e\in H_{r}^{1}\left( \mathbb{R}^{3}\right) $ such that
\begin{equation*}
\int_{\mathbb{R}^{3}}F\left( e\right) dx-\frac{1}{2}\left\Vert e\right\Vert
_{H^{1}}^{2}>0.
\end{equation*}%
If $2<q<3,$ then by $a_{q}>0$ and Fatou's lemma, for $u\in H_{r}^{1}\left(
\mathbb{R}^{3}\right) \setminus \left\{ 0\right\} ,$ we have%
\begin{equation*}
\lim_{t\rightarrow \infty }\frac{1}{t^{q}}\left[ \frac{1}{2}\left\Vert
tu\right\Vert _{H^{1}}^{2}-\int_{\mathbb{R}^{3}}F\left( tu\right) dx\right]
=-a_{q}\int_{\mathbb{R}^{3}}\left\vert u\right\vert ^{q}dx<0,
\end{equation*}%
and so there exists $\widehat{e}\in H_{r}^{1}\left( \mathbb{R}^{3}\right) $
such that
\begin{equation*}
\int_{\mathbb{R}^{3}}F\left( \widehat{e}\right) dx-\frac{1}{2}\left\Vert
\widehat{e}\right\Vert _{H^{1}}^{2}>0.
\end{equation*}%
This implies that $\mathbf{A}_{0}$ is nonempty.\newline
$\left( ii\right) $ For each $u\in \mathbf{A}_{0}$ there exists $\lambda
^{\ast }>0$ such that%
\begin{equation*}
\frac{1}{2}\left\Vert u\right\Vert _{H^{1}}^{2}+\frac{\lambda ^{\ast }}{4}%
\int_{\mathbb{R}^{3}}\phi _{u}u^{2}dx-\int_{\mathbb{R}^{3}}F\left( u\right)
dx<0
\end{equation*}%
or%
\begin{equation*}
\frac{\lambda ^{\ast }}{4}<\frac{\int_{\mathbb{R}^{3}}F\left( u\right) dx-%
\frac{1}{2}\left\Vert u\right\Vert _{H^{1}}^{2}}{\int_{\mathbb{R}^{3}}\phi
_{u}u^{2}dx},
\end{equation*}%
indicating that there exists $\widehat{\lambda }^{\ast }>0$ such that $%
\Lambda _{0}\geq \widehat{\lambda }^{\ast }.$ Next, we show that there
exists $M_{0}>0$ such that $\Lambda _{0}\leq M_{0}.$ By conditions $\left(
F1\right) $ and $\left( F2\right) ,$ there exists\thinspace $C_{1}>0$ such
that
\begin{equation}
F\left( u\right) \leq \frac{1}{2}u^{2}+C_{1}\left\vert u\right\vert ^{3}.
\label{A-1}
\end{equation}%
Since
\begin{eqnarray*}
C_{1}\int_{\mathbb{R}^{3}}\left\vert u\right\vert ^{3}dx &=&C_{1}\int_{%
\mathbb{R}^{3}}\left( -\Delta \phi _{u}\right) \left\vert u\right\vert
dx=C_{1}\int_{\mathbb{R}^{3}}\left\langle \nabla \phi _{u},\nabla \left\vert
u\right\vert \right\rangle dx \\
&\leq &\frac{1}{2}\int_{\mathbb{R}^{3}}\left\vert \nabla u\right\vert ^{2}dx+%
\frac{C_{1}^{2}}{2}\int_{\mathbb{R}^{3}}\left\vert \nabla \phi
_{u}\right\vert ^{2}dx \\
&=&\frac{1}{2}\int_{\mathbb{R}^{3}}\left\vert \nabla u\right\vert ^{2}dx+%
\frac{C_{1}^{2}}{2}\int_{\mathbb{R}^{3}}\phi _{u}u^{2}dx\text{ for all }u\in
H^{1}(\mathbb{R}^{3}),
\end{eqnarray*}%
by $\left( \ref{A-1}\right) $, we have%
\begin{eqnarray*}
\frac{\int_{\mathbb{R}^{3}}F\left( u\right) dx-\frac{1}{2}\left\Vert
u\right\Vert _{H^{1}}^{2}}{\int_{\mathbb{R}^{3}}\phi _{u}u^{2}dx} &\leq &%
\frac{C_{1}^{2}}{2}\times \frac{\frac{1}{2}\int_{\mathbb{R}%
^{3}}u^{2}dx+C_{1}\int_{\mathbb{R}^{3}}\left\vert u\right\vert ^{3}dx-\frac{1%
}{2}\left\Vert u\right\Vert _{H^{1}}^{2}}{C_{1}\int_{\mathbb{R}%
^{3}}\left\vert u\right\vert ^{3}dx-\frac{1}{2}\int_{\mathbb{R}%
^{3}}\left\vert \nabla u\right\vert ^{2}dx} \\
&\leq &\frac{C_{1}^{2}}{2}\times \frac{C_{1}\int_{\mathbb{R}^{3}}\left\vert
u\right\vert ^{3}dx-\frac{1}{2}\int_{\mathbb{R}^{3}}\left\vert \nabla
u\right\vert ^{2}dx}{C_{1}\int_{\mathbb{R}^{3}}\left\vert u\right\vert
^{3}dx-\frac{1}{2}\int_{\mathbb{R}^{3}}\left\vert \nabla u\right\vert ^{2}dx}
\\
&=&\frac{C_{1}^{2}}{2}.
\end{eqnarray*}%
Thus,%
\begin{equation*}
0<\Lambda _{0}:=\sup_{u\in \mathbf{A}_{0}}\frac{\int_{\mathbb{R}^{3}}F\left(
u\right) dx-\frac{1}{2}\left\Vert u\right\Vert _{H^{1}}^{2}}{\int_{\mathbb{R}%
^{3}}\phi _{u}u^{2}dx}\leq \frac{C_{1}^{2}}{2}.
\end{equation*}%
This completes the proof.
\end{proof}

\begin{theorem}
\label{t6-2}Suppose that conditions $\left( F1\right) $ and $\left(
F2\right) $ hold. Then we have\newline
$\left( i\right) $ $\overline{\mathbf{A}}_{0}$ is a nonempty set.\newline
$\left( ii\right) $ There exists $\overline{M}_{0}>0$ such that $0<\overline{%
\Lambda }_{0}\leq \overline{M}_{0}.$
\end{theorem}

\begin{proof}
The proof is similar to the argument used in Theorem \ref{t6-1} and is
omitted here.
\end{proof}

\section*{Acknowledgments}

The second author was supported in part by the Ministry of Science and
Technology, Taiwan (Grant No. 110-2115-M-390-006-MY2).

\end{document}